\documentclass[dvipsnames, a4paper, 12pt]{article}
\usepackage[utf8]{inputenc}

\usepackage{amsthm}
\usepackage[most]{tcolorbox}

\newtheoremstyle{upright}   %
  {10pt}                    %
  {10pt}                    %
  {}                        %
  {}                        %
  {\bfseries}               %
  {.}                       %
  { }                       %
  {}                        %

\theoremstyle{upright}

\tcbuselibrary{theorems}
\definecolor{theoremblue}{HTML}{2196F3} %

\tcolorboxenvironment{theorem}{
  enhanced jigsaw,
  sharp corners,
  colframe=theoremblue,
  colback=theoremblue!5,
  coltitle=black,
  fonttitle=\bfseries,
  borderline west={2pt}{0pt}{theoremblue},
  before skip=10pt,
  after skip=10pt,
  boxrule=0pt,
  left=10pt,
  right=10pt,
  breakable
}

\tcolorboxenvironment{definition}{
  enhanced jigsaw,
  sharp corners,
  colframe=theoremblue,
  colback=theoremblue!5,
  coltitle=black,
  fonttitle=\bfseries,
  borderline west={2pt}{0pt}{theoremblue},
  before skip=10pt,
  after skip=10pt,
  boxrule=0pt,
  left=10pt,
  right=10pt,
  breakable
}

\tcolorboxenvironment{corollary}{
  enhanced jigsaw,
  sharp corners,
  colframe=theoremblue,
  colback=theoremblue!5,
  coltitle=black,
  fonttitle=\bfseries,
  borderline west={2pt}{0pt}{theoremblue},
  before skip=10pt,
  after skip=10pt,
  boxrule=0pt,
  left=10pt,
  right=10pt,
  breakable
}

\tcolorboxenvironment{prop}{
  enhanced jigsaw,
  sharp corners,
  colframe=theoremblue,
  colback=theoremblue!5,
  coltitle=black,
  fonttitle=\bfseries,
  borderline west={2pt}{0pt}{theoremblue},
  before skip=10pt,
  after skip=10pt,
  boxrule=0pt,
  left=10pt,
  right=10pt,
  breakable
}

\tcolorboxenvironment{condition}{
  enhanced jigsaw,
  sharp corners,
  colframe=theoremblue,
  colback=theoremblue!5,
  coltitle=black,
  fonttitle=\bfseries,
  borderline west={2pt}{0pt}{theoremblue},
  before skip=10pt,
  after skip=10pt,
  boxrule=0pt,
  left=10pt,
  right=10pt,
  breakable
}

\tcolorboxenvironment{example}{
  enhanced jigsaw,
  sharp corners,
  colframe=theoremblue,
  colback=theoremblue!5,
  coltitle=black,
  fonttitle=\bfseries,
  borderline west={2pt}{0pt}{theoremblue},
  before skip=10pt,
  after skip=10pt,
  boxrule=0pt,
  left=10pt,
  right=10pt,
  breakable
}

\tcolorboxenvironment{lemma}{
  enhanced jigsaw,
  sharp corners,
  colframe=theoremblue,
  colback=theoremblue!5,
  coltitle=black,
  fonttitle=\bfseries,
  borderline west={2pt}{0pt}{theoremblue},
  before skip=10pt,
  after skip=10pt,
  boxrule=0pt,
  left=10pt,
  right=10pt,
  breakable
}

\tcolorboxenvironment{assumption}{
  enhanced jigsaw,
  sharp corners,
  colframe=theoremblue,
  colback=theoremblue!5,
  coltitle=black,
  fonttitle=\bfseries,
  borderline west={2pt}{0pt}{theoremblue},
  before skip=10pt,
  after skip=10pt,
  boxrule=0pt,
  left=10pt,
  right=10pt,
  breakable
}

\tcolorboxenvironment{remark}{
  enhanced jigsaw,
  sharp corners,
  colframe=theoremblue,
  colback=theoremblue!5,
  coltitle=black,
  fonttitle=\bfseries,
  borderline west={2pt}{0pt}{theoremblue},
  before skip=10pt,
  after skip=10pt,
  boxrule=0pt,
  left=10pt,
  right=10pt,
  breakable
}

\usepackage{hyperref}
\usepackage[utf8]{inputenc}
\usepackage[T1]{fontenc}
\usepackage{amsmath}
\usepackage{amsthm}
\usepackage{amsfonts}
\usepackage{amssymb}
\usepackage{graphicx,caption,subcaption}
\usepackage{fancybox}
\usepackage{enumitem}
\usepackage{soul}
\usepackage{mathtools}
\usepackage{lmodern}
\usepackage{stmaryrd}
\usepackage{multirow}
\usepackage{multicol}
\usepackage{xcolor}
\usepackage{tabularx}
\usepackage{array}
\usepackage{colortbl}
\usepackage{varwidth}
\PassOptionsToPackage{dvipsnames,svgnames}{xcolor}     
\usepackage{wrapfig}
\usepackage[explicit,calcwidth]{titlesec}
\usepackage{fancyhdr}
\usepackage[left=2.1cm, right=2.1cm, top=2cm, bottom=2cm]{geometry}
\usepackage{esint}
\usepackage{tocloft}
\usepackage[skip=8pt plus1pt, indent=0pt]{parskip}
\usepackage{twoopt}
\usepackage{authblk}
\usepackage{accents}
\usepackage{mathrsfs}
\usepackage{tikz-cd}
\usetikzlibrary{matrix,arrows}
\usepackage{adjustbox} %
\hypersetup{ colorlinks=true, %
    linktoc=all,     %
    linkcolor=blue,  %
} \PassOptionsToPackage{unicode}{hyperref}
\PassOptionsToPackage{naturalnames}{hyperref}
\usepackage{float}
\usepackage[style=american]{csquotes}

\usepackage[ruled,vlined,linesnumbered]{algorithm2e}
\usepackage{todonotes}
\setuptodonotes{fancyline, color=red!20, shadow, inline}

\usepackage{silence}

\usepackage[bbgreekl]{mathbbol}
\DeclareSymbolFont{bbold}{U}{bbold}{m}{n}
\DeclareSymbolFontAlphabet{\mathbbold}{bbold}
\DeclareSymbolFontAlphabet{\mathbb}{AMSb}%

\usepackage[nameinlink, capitalise]{cleveref}

\newtheorem{theorem}{Theorem}
\newtheorem*{theorem*}{Theorem}
\newtheorem{corollary}[theorem]{Corollary}
\newtheorem{lemma}[theorem]{Lemma}
\newtheorem*{lemma*}{Lemma}

\newtheorem{prop}[theorem]{Proposition}

\newtheorem{definition}[theorem]{Definition}

\newcommand{\N}[0]{\mathbb{N}}

\newcommand{\R}[0]{\mathbb{R}}
\newcommand{\C}[0]{\mathbb{C}}

\newcommand{\V}[1]{\mathbb{V}\left[#1\right]}

\newcommand{\dd}[0]{\mathrm{d}}

\newcommand{\Span}[0]{\mathrm{Span}}

\DeclareMathOperator{\argmin}{argmin}

\DeclareMathOperator{\supp}{supp}
\DeclareMathOperator{\diam}{diam}

\newcommand{\app}[4]{\left\lbrace\begin{array}{ccc}
   #1 & \longrightarrow & #2 \\
   #3 & \longmapsto & #4 \\
\end{array} \right.}

\newcommand{\oll}[1]{\overline{#1}}

\newcommand{\X}{\mathcal{X}}
\newcommand{\Y}{\mathcal{Y}}

\newcommand{\diamX}{D_\X}
\newcommand{\diamW}{D_W}
\newcommand{\radW}{R_W}

\definecolor{mybluei}{HTML}{ABE6FF}

\newcommand{\ukj}[1]{{\color{Fuchsia}#1}}

\usepackage[
    backend=biber,
    uniquename=false,
    bibencoding=utf8,
    sorting=nyt,
    doi=false, isbn=false, url=false,
    style=alphabetic,
    maxcitenames=2,
    uniquelist=false,
    maxbibnames=123,
    backref=false,  %
    hyperref=true
]{biblatex}
\title{Explicit Universal and Approximate-Universal Kernels on Compact Metric Spaces}
\author[1]{Eloi Tanguy}
\affil[1]{Universit\'e Paris Cit\'e, CNRS, MAP5, F-75006 Paris, France}
\date{\today}

\renewcommand{\ukj}[1]{#1}
  
\begin{document}
\maketitle

\begin{abstract}
    Universal kernels, whose Reproducing Kernel Hilbert Space is dense in the space of continuous functions are of great practical and theoretical interest. In this paper, we introduce an explicit construction of universal kernels on compact metric spaces. We also introduce a notion of approximate universality, and construct tractable kernels that are approximately universal.
\end{abstract}

\tableofcontents
\newpage

\section{Introduction}

\subsection{Kernels in Practice and Related Works}

The theory of Reproducing Kernel Hilbert Spaces (RKHS) was introduced by
Aronszajn \cite{aronszajn1950theory} and has been the object of several
monographs \cite{scholkopf2002learning,christmann2008support,saitoh2016theory}.
In Machine Learning, the first practical uses of RKHSs hinged on the ``kernel
trick'' introduced by
\cite{aizerman1964theoretical,boser1992training,scholkopf1998nonlinear}, which
allows high expressivity of models without the need of an explicit feature map
into the underlying infinite-dimensional space. A cornerstone model is the
Support Vector Machine \cite{cortes1995support}, whose statistical properties
have garnered extensive attention, see for example the monograph
\cite{christmann2008support}. A useful tool is the Kernel Mean embedding (we
refer to the review \cite{muandet2017kernel}) which maps a measure $\mu$ to a
point $M(\mu)$ in a Hilbert space of features, and can be used to compare
measures with the Maximum Mean Discrepancy defined as $\mathrm{MMD}(\mu, \nu) =
\|M(\mu)-M(\nu)\|_H$ which fostered numerous applications
\cite{gomez2009mean,zhang2011kernel,muandet2012learning,fukumizu2013kernel,gretton2012kernel,doran2014permutation,li2017mmd}.
Theoretical guarantees for the MMD depending on the kernel have been reviewed in
\cite{sriperumbudur2011universality}. RKHSs and kernels in general were also
vastly used in statistics and probability, we refer to the monograph
\cite{berlinet2011reproducing} for an overview.

Some theoretical questions remain open, in particular constructing suitable
kernels on non-euclidean metric spaces is a challenging problem that is the
subject of ongoing research. For compact metric spaces,
\cite{christmann2010universal} show the existence of universal kernels (i.e.
such that the associated RKHS is dense in the space of continuous functions)
when the space is continuously embedded into a separable Hilbert space, and
\cite{steinwart2021strictly} relate the notions of universality and strictly
proper kernel scores. On complete Riemannian manifolds,
\cite{jayasumana2015kernel} observe (Theorem 6.2) that the natural Gaussian
kernel $k(x,y) = \exp(-s\ d(x,y)^2)$ is indeed a kernel only in the very
restricted case where the manifold is isometric to $\R^d$. On Hilbert and Banach
spaces, \cite{ziegel2024characteristic} introduce radial kernels and show
universality-adjacent properties. Regarding universality,
\cite{micchelli2006universal} study conditions on the feature maps that ensure
universality.

Our contribution first consists in an \textit{explicit} construction of
universal kernels on a compact metric space $(\X, d_\X)$, in some sense
extending \cite{christmann2010universal} whose construction is not explicit and
relied on the existence of an embedding. The constructed kernels use known
kernels known as \textit{Taylor} and \textit{radial} kernels, which are defined
on compact subsets of separable Hilbert spaces. Noticing that our kernels are
not tractable in practice, we introduce a notion of \textit{approximate
universality} and construct other explicit kernels that are approximately
universal and tractable.

\subsection{Elements of RKHS Theory}
For a set $\X$, a \textit{kernel} $k: \X\times\X\longrightarrow \R$ is a
\textit{positive-definite symmetric} function, which is to say a function that
verifies $k(x,y)=k(y,x)$ and:
$$\forall n\in \N^*,\; \forall (x_1, \cdots, x_n) \in \X^n,\; \forall a\in
\R^n,\; \sum_{i=1}^n\sum_{j=1}^n a_ik(x_i, x_j)a_j \geq 0. $$ By the
Moore-Aronszajn theorem \cite{aronszajn1950theory}, there exists a unique
Hilbert space $(H, \langle \cdot, \cdot \rangle_H)$ of functions
$\X\longrightarrow\R$, such that $H$ contains all basic functions $k(\cdot, x)$,
and its inner product is characterised by the ``reproducing property'' $\langle
k(\cdot, x), k(\cdot, y) \rangle_H = k(x, y)$. Denoting by $\oll{\Span}$ the
Hilbertian completion of the linear span of a set, it follows that $H =
\oll{\Span}\{k(\cdot, x),\; x \in \X\}$. The space $H$ is referred to as the
Reproducing Kernel Hilbert Space (RKHS) associated to the kernel $k$. The
reproducing property of the kernel implies that for any $h\in H$ and $x\in \X$,
we have $\langle h, k(\cdot, x)\rangle_H = h(x)$.

If $k$ is continuous (w.r.t. a metric on $\X$), the RKHS $H$ is contained in the
space of continuous functions from $\X$ to $\R$, denoted $\mathcal{C}(\X)$. In
this work, we will always consider continuous kernels. Some continuous kernels
have an additional property called \textit{universality}:
\begin{definition}\label{def:universal_kernel} A continuous kernel $k$ on a
    compact metric space $(\X, d_\X)$ is said to be \textit{universal} if the
    RKHS $H$ is dense in $(\mathcal{C}(\X), \|\cdot\|_{\infty})$, the space of
    continuous functions from $\X$ to $\R$ equipped with the supremum norm. In
    other words, for any $\varepsilon>0$ and $f\in \mathcal{C}(\X)$, there
    exists $h\in H$ such that $\|f-h\|_\infty \leq \varepsilon$.
\end{definition}

Another equivalent definition of kernels uses the notion of feature map /
feature space pairs: through these lens, a kernel is any map $\X^2
\longrightarrow \R$ such that there exists a Hilbert space $H_0$ and a map
$\Phi_0: \X\longrightarrow H_0$ such that \cref{eqn:def_feature} holds.
\begin{equation}\label{eqn:def_feature}
    \forall x, y \in \X,\; k(x,y) = \langle \Phi_0(x), \Phi_0(y) \rangle_{H_0}.
\end{equation}
The pair ($\Phi_0, H_0$) is called a \textit{feature map / feature space pair}
(or simply \textit{feature pair}) for $k$, and any kernel can be written in this
form \cite[Theorem 4.16]{christmann2008support}. The associated RKHS is then
defined as:
\begin{equation}\label{eqn:def_feature_RKHS}
    H = \{x\longmapsto \langle h_0, \Phi_0(x)\rangle_{H_0},\; h_0\in H_0\}.
\end{equation}
The RKHS $H$ in \cref{eqn:def_feature_RKHS} is unique \cite[Theorem
4.21]{christmann2008support}, and equal to $\oll{\Span}\left\{k(\cdot, x),\;
x\in \X\right\}$ as stated above. The \textit{canonical feature map} is defined
as $\Phi(x) = k(\cdot, x)$, and the pair ($\Phi, H$) is called the
\textit{canonical feature pair} for $k$.

From the space viewpoint, an RKHS can equivalently be defined as a Hilbert space
of functions $\X \longrightarrow \R$ in which the evaluation $\delta_x: h
\longmapsto h(x)$ is continuous for all $x\in \X$, as is done in \cite[Section
4.2]{christmann2008support}. The kernel is then defined as $k(x,y) = \langle
L\delta_x, L\delta_y \rangle_H$, where $L\delta_x \in H$ is the Riesz
representation of $\delta_x \in H'$. In this paper, we stick to the (equivalent)
kernel viewpoint.

For a compact metric space $(E, d_E)$, we will denote by $\diam(E)$ its
diameter, which is defined by $\diam(E) := \max_{(x,y)\in E^2}d_E(x,y)$.
Throughout this work, $\X$ will be assumed to be a compact metric space, and we
denote $\diamX := \diam(\X)$.

The first type of universal kernels of interest in this work are Taylor kernels
(see \cite[Lemma 4.8 and Corollary 4.57]{christmann2008support} for their study
on compact subsets of $\R^d$).
\begin{definition}\label{def:taylor_kernel} Let $W\subset \ell^2$ be a non-empty
    compact set and $\diamW^2:=\diam(W)^2>0$ the square of its diameter. Take a
    sequence $(a_n)_{n\in \N} \in (0, +\infty)^\N$ such that $K(t):=\sum_na_n
    t^n$ converges absolutely on $[-\diamW^2, \diamW^2]$. The Taylor kernel
    associated to $K$ is the map
    \begin{equation}
        k_W := \app{W^2}{\R}{(u,v)}{K(\langle u, v \rangle_{\ell^2})}.
    \end{equation}
\end{definition}
Taylor kernels are shown to be universal on compact subsets of $\ell^2$ in
\cite[Theorem 2.1]{christmann2010universal}. The second type of universal
kernels we will consider are radial kernels\footnote{Radial kernels can be
defined (and shown to be universal) on separable Hilbert spaces and more
\cite{ziegel2024characteristic}, but we will use compactness for other reasons,
and thus restrict to compact subsets of $\ell^2$ for our purposes.}.
\begin{definition}\label{def:radial_kernel} Let $W\subset \ell^2$ be a non-empty
    compact set and $\mu\in \mathcal{M}([0,+\infty))$ a finite Borel measure on
    $[0, +\infty)$ with $\supp(\mu)\neq \{0\}$. The associated radial function
    $K$ and the radial kernel $k_W$ are defined as follows:
    \begin{equation}
        K := \app{\R_+}{\R}{t}{\int_0^{+\infty}e^{-st}\dd\mu(s)},\; 
        k_W := \app{W^2}{\R}{(u,v)}{K(\|u-v\|_{\ell^2}^2)}.
    \end{equation}
\end{definition}
The universality of radial kernels on $W$ is a consequence of \cite[Proposition
5.2]{ziegel2024characteristic} combined with \cite[Theorem
3.13]{steinwart2021strictly}. Note that the well-known Gaussian (or RBF) kernel
$\exp(-\|\cdot-\cdot\|_{\ell^2}^2/(2\sigma^2))$ is a particular radial kernel
with $\mu := \delta_{1/(2\sigma^2)}$.

\subsection{Paper Outline and Contributions}

The objective of this paper is to construct kernels $k$ on a compact metric
space $(\X, d_\X)$ that are \textit{universal} (see
\cref{def:universal_kernel}). We also introduce a notion of approximate
universality (\cref{def:approximate_universal_kernel}), and introduce other
(tractable) explicit kernels $\hat k$ and $k_t$ that verify this property.

\paragraph{Construction of universal kernels in \cref{sec:univ_kernel}} To
construct universal kernels on $\X$, we first introduce an explicit continuous
injection $\varphi: \X \longrightarrow \ell^2$ in \cref{prop:inj_X_l2}. Given
any universal kernel $k_V$ on $V :=\varphi(\X)\subset\ell^2$ we show in
\cref{thm:univ_kernel_X} that $k(x,y) := k_V(\varphi(x), \varphi(y))$, is
universal on $\X$.

The construction of $\varphi$ in \cref{sec:univ_kernel} is based on a countable
basis of $\X$, and the associated kernel requires inner products in $\ell^2$. In
\cref{sec:approx_univ_kernels}, we explain how we can use instead a (finite)
$\eta$-covering of $\X$, yielding a finite-dimensional approximation of the
embedding $\varphi$, with theoretical guarantees. We also investigate the
natural idea of truncating the sequence $\varphi(x)$.

\paragraph{Approximate universal kernels in \cref{sec:approx_univ_kernels}} We
introduce a notion of \textit{approximate universal kernels} on $\X$, which are
kernels $\hat k$ of RKHS $\hat H$ such that for all $\varepsilon>0$ and $f\in
\mathcal{C}(\X)$, there exists $\hat h\in \hat H$ such that $\|f-\hat h\|_\infty
\leq \varepsilon + \rho(f)$, where $\rho(f)>0$ is an error term depending on
$\hat k$ and $f$. We construct a simpler map $\hat \varphi: \X \longrightarrow
\R^J$ as a surrogate for the embedding $\varphi: \X: \longrightarrow \ell^2$,
and embed $\R^J$ into $\ell^2$ appropriately to compare $\varphi$ and
$B\circ\hat\varphi$. This allows us to introduce the kernel $\hat k(x,y) :=
k_W(B\circ\hat\varphi(x), B\circ\hat\varphi(y))$ for a compact set $W \supset
\varphi(\X)\cup B(\hat\varphi(\X))$ and a Taylor or radial kernel $k_W$ on $W$.
In \cref{cor:tractable_hat_k}, we provide a tractable (as in numerically
computable) expression for $\hat k$. Finally, we show in
\cref{thm:approx_univ_kernel} that $\hat k$ is an approximate universal kernel
on $\X$ with an explicit error term $\rho$ depending on discretisation
parameters and the ``complexity'' of the function $f$. In
\cref{sec:truncated_kernels}, we introduce a simple truncation of $\varphi$
which leads to another approximate universal kernel $k_t$ on $\X$.
\newpage
\section{Explicit Universal Taylor and Radial Kernels on a Compact Metric
Space}\label{sec:univ_kernel}

\ukj{
As a preliminary to our main constructions, we begin with two elementary general
properties of RKHS which will be useful throughout this section. We remind that
a \textit{homeomorphism} is a continuous bijection with a continuous inverse.

\begin{lemma}\label{lemma:preliminary_kernels}
    \begin{itemize}
        \item[i)] Let $\X, \Y$ be two sets and $k_\Y :\Y^2 \longrightarrow \R$ a
        kernel on $\Y$, and $\varphi: \X \longrightarrow \Y$ a function. The map
        $k_\X$ defined in \cref{eqn:composition_kernel} is a kernel on $\X$:
        \begin{equation}\label{eqn:composition_kernel}
            k_\X := \app{\X^2}{\R}{(x, x')}{k_\Y(\varphi(x), \varphi(x'))}.
        \end{equation}

        \item[ii)] If additionally $\X$ and $\Y$ are compact metric spaces,
        $\varphi$ is a homeomorphism and $k_\Y: \Y^2 \longrightarrow \R$ is
        universal, the kernel $k_\X$ is universal.
    \end{itemize}
\end{lemma}
\begin{proof}
    \textbf{Proof of i):} We verify immediately using the definition that $k_\X$
    is a positive-definite symmetric function on $\X$ using the fact that $k_\Y$
    is such on $\Y$.

    \textbf{Proof of ii):} Let $H_\Y$ be the unique RKHS associated to the
    kernel $k_\Y$ on $\Y$, and $\Phi_\Y: \Y \longmapsto H_\Y$ its canonical
    feature map (i.e. $\forall y \in \Y,\; \Phi_\Y(y) = k_\Y(\cdot, y)$). Since
    $k_\X(x,x') = \langle \Phi_\Y\circ\varphi(x), \Phi_\Y\circ\varphi(x')
    \rangle_{H_\Y}$, the map $\Phi_\Y\circ\varphi$ and the space $H_\Y$ are a
    feature pair for $k_\X$. In the following, given a set $\mathcal{F}$ of
    functions and $g$ a function, we write $\mathcal{F}\circ g := \{f\circ g,\;
    f \in \mathcal{F}\}$. By uniqueness \cite[Theorem
    4.21]{christmann2008support}, it follows that the RKHS $H_\X$ associated to
    $k_\X$ can be written
    $$H_\X = \{x\longmapsto \langle h_\Y,
    \Phi_\Y\circ\varphi(x)\rangle_{H_\Y},\; h_\Y\in H_\Y\}= H_\Y \circ
    \varphi,$$ where the second equality comes from the reproducing property:
    for any $x\in\X$ and $h_\Y\in H_\Y$, we have $h_\Y\circ\varphi(x) = \langle
    h_\Y, \Phi_\Y\circ\varphi(x)\rangle_{H_\Y}$. Since $\varphi$ is a
    homeomorphism, we also have $\mathcal{C}(\X) =\mathcal{C}(\Y)\circ\varphi$.
    Now for $\varepsilon>0$ and $f_\X\in \mathcal{C}(\X)$, take $f_\Y :=
    f_\X\circ \varphi^{-1}\in\mathcal{C}(\Y)$. By universality of $k_\Y$, there
    exists $h_\Y\in H_\Y$ such that $\|f_\Y - h_\Y\|_{\infty} \leq \varepsilon$.
    Taking $h_\X := h_\Y\circ\varphi \in H_\X$ yields $\|f_\X-h_\X\|_\infty \leq
    \varepsilon$, and as a result $k_\X$ is universal.
\end{proof}

\cref{lemma:preliminary_kernels} is useful for the construction of universal
kernels on compact metric spaces $\X$ using universal kernels on another space.
In the following, we will consider a space $\Y$ which is a compact subspace of
the Hilbert space $\ell^2$ of square-summable sequences.
}

\subsection{Injection of \texorpdfstring{$\X$}{S} into \texorpdfstring{$\ell^2$}{l2}}\label{sec:inj_X_l2}

Let $(\X, d_\X)$ be a non-empty compact metric space, and let $\diamX > 0$ be
its diameter. We take a \textit{basis} of $\X$, i.e. a countable sequence
$(x_n)_{n\in \N}$ such that for any $x\in \X$ and $\varepsilon>0$, there exists
$n\in \N$ such that $d_\X(x, x_n) \leq \varepsilon$. Using a basis, we construct
an implicit continuous injection $\varphi$ from $\X$ into $\ell^2$
(\cref{prop:inj_X_l2}), then use universal kernels on $V := \varphi(\X)$ to
build a universal kernel $k$ on $\X$ in \cref{thm:univ_kernel_X}. In
\cref{fig:phi}, we illustrate the injection.
\begin{figure}[H]
    \centering
    \includegraphics[width=0.8\linewidth]{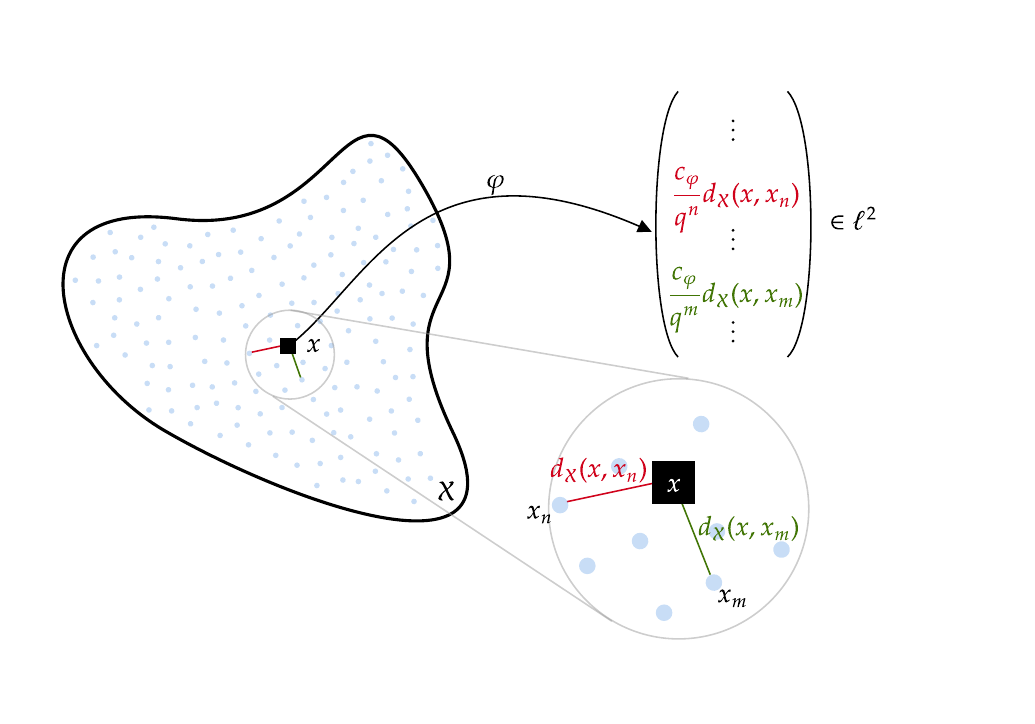}
    \caption{Given a basis $(x_n)_{n\in \N}$ of $\X$, the mapping $\varphi: \X
    \longrightarrow \ell^2$ maps a point $x \in \X$ to the sequence of its
    distances to the points of the basis.}
    \label{fig:phi}
\end{figure}

\begin{prop}\label{prop:inj_X_l2} Let $(x_n)$ a basis of $\X$ and $q>1$. The map
    $$\varphi := \app{\X}{\ell^2}{x}{\left(\cfrac{c_\varphi d_\X(x,
    x_n)}{q^n}\right)_{n\in \N}},\quad c_\varphi := \frac{\sqrt{q^2 -1}}{q}$$ is
    1-Lipschitz and injective.
\end{prop}
\begin{proof}
    The fact that $\varphi(\X) \subset \ell^2$ comes from the compactness of
    $\X$. Take now $x, y \in \X$:
    \begin{align*}
        \|\varphi(x) - \varphi(y)\|_{\ell^2}^2 
        = c_\varphi^2\sum_{n=0}^{+\infty}
        \cfrac{|d_\X(x, x_n) - d_\X(y, x_n)|^2}{q^{2n}}
        \leq c_\varphi^2\sum_{n=0}^{+\infty}\cfrac{d_\X(x, y)^2}{q^{2n}}
        = \cfrac{c_\varphi^2q^2}{q^2 - 1}d_\X(x, y)^2,
    \end{align*}
    showing 1-Lipschitzness. As for injectivity, consider $x\neq y \in \X^2$ and
    $\varepsilon := d_\X(x, y) / 3 > 0$. Since $(x_n)$ is a basis of $\X$, there
    exists $n\in \N$ such that $d_\X(x, x_n) \leq \varepsilon$. The triangle
    inequality then shows
    $$d_\X(y,x_n) \geq \underbrace{d_\X(y, x)}_{=3\varepsilon} -
    \underbrace{d_\X(x, x_n)}_{\in [0, \varepsilon]} \geq 2\varepsilon, $$ and
    thus $|\underbrace{d_\X(y, x_n)}_{\geq 2\varepsilon} - \underbrace{d_\X(x,
    x_n)}_{\in [0, \varepsilon]}|\geq \varepsilon$, allowing us to conclude
    $$\|\varphi(y) - \varphi(x)\|_{\ell^2}^2 \geq c_\varphi^2\cfrac{|d_\X(y,
    x_n) - d_\X(x, x_n)|^2}{q^{2n}} \geq
    c_\varphi^2\cfrac{\varepsilon^2}{q^{2n}} > 0.$$
\end{proof}

\subsection{Universal Kernels on \texorpdfstring{$\X$}{X}}\label{sec:univ_kernels}

We can now build universal kernels $k: \X^2\longrightarrow \R$ using $\varphi$
and a universal kernel \ukj{$k_W: W^2 \longrightarrow \R$ (for example Taylor or
radial) on $W$ a compact subset of $\ell^2$ containing $V := \varphi(\X)$}. The
technique follows closely that of \cite[Theorem 2.2]{christmann2010universal}.
Note that thanks to the 1-Lipschitzness of $\varphi$, we have
$\diam(\varphi(\X)) \leq \diam(\X) =: \diamX$.
\begin{theorem}\label{thm:univ_kernel_X} Let $V := \varphi(\X)\subset \ell^2$
    and \ukj{$W$ be a compact subset of $\ell^2$ containing $V$. Consider $k_W:
    W^2 \longrightarrow \R$ a universal kernel on $W$} (e.g. Taylor as in
    \cref{def:taylor_kernel} or radial as in \cref{def:radial_kernel}). The
    kernel
    $$k := \app{\X^2}{\R}{(x,y)}{k_W(\varphi(x), \varphi(y))} $$ is universal on
    $\X$.
\end{theorem}
\begin{proof}
    \ukj{We introduce $k_V: V^2 \longrightarrow \R$ the restriction of $k_W$ to
    $V^2$. By \cite[Lemma 4.55 item iii)]{christmann2008support}, $k_V$ remains
    universal.} Since $\X$ is a compact metric space and $\ell^2$ is Hausdorff,
    the co-restriction of $\varphi$ to $V$ denoted $\varphi_V: \X
    \longrightarrow V$ is a homeomorphism. \ukj{Noticing that $k = (x, y) \in
    \X^2 \longmapsto k_V(\varphi_V(x), \varphi_V(y))$, by
    \cref{lemma:preliminary_kernels}, $k$ is thus a universal kernel on $\X$.}
\end{proof}

\paragraph{A strictly convex functional on $\mathcal{P}(\X)$} We consider the
set $\mathcal{P}(\X)$ of probability measures on $\X$. As a universal kernel,
$k$ is also \textit{characteristic} (see \cite{sriperumbudur2011universality}
and use the compactness of $\X$), which is to say that the map
$$M := \app{\mathcal{P}(\X)}{H}{\mu}{\int_\X k(\cdot, x) \dd\mu(x)}, $$ known as
the \textit{kernel mean embedding} \cite{sriperumbudur2010hilbert}, is
injective. One can show that the map
$$F := \app{\mathcal{P}(\X)}{\R_+}{\mu}{\|M(\mu)\|_H^2} $$ is continuous with
respect to the weak convergence of measures (apply \cite[Theorem
A.1]{hable2011qualitative} using that $x\longmapsto k(\cdot, x)$ is continuous
and bounded). Furthermore, by linearity of $M$ and strict convexity of
$\|\cdot\|_H^2$, the function $F$ is strictly convex \ukj{with respect to the
``vertical'' convex combination of probability measures: 
$$\forall \mu, \nu \in \mathcal{P}(\X),\; \forall t\in (0, 1),\; F((1-t)\mu +
t\nu) < (1-t)F(\mu) + tF(\nu).$$} Note that the fact that $M$ is injective is
required to prove \textit{strict} convexity.
\newpage
\section{Approximate Universal Kernels}\label{sec:approx_univ_kernels}

In practice, the function $\varphi$ introduced in \cref{prop:inj_X_l2} is not
tractable \ukj{(in the sense that computing and storing a full sequence
$\varphi(x)\in \ell^2$ is numerically impossible)}, limiting the use of the
kernels proposed in \cref{thm:univ_kernel_X} \ukj{in their exact formulation}.
\ukj{The natural idea of truncating the sequence $\varphi(x)$ to the first $N$
terms is tackled later in \cref{sec:truncated_kernels}, we begin with the main
approach of the paper, which relies on a discretisation of the space $\X$.}

We will now introduce a family of tractable kernels which are approximately
universal on $\X$. Throughout this section, the kernels $k_W$ on a compact
subset $W$ of $\ell^2$ that we will consider are Taylor or radial (see
\cref{def:taylor_kernel,def:radial_kernel}). Our objective is to construct
another kernel $\hat k$ with a simpler explicit mapping $\hat \varphi:
\X\longrightarrow \R^J$, yielding an RKHS $\hat H$ which we will show to be
approximately universal in the sense of \cref{def:approximate_universal_kernel}.
\begin{definition}\label{def:approximate_universal_kernel} Let $\hat{k}:
    \X^2\longrightarrow \R$ be a kernel on $\X$ of RKHS $\hat H$ and $\rho:
    \mathcal{C}(\X) \longrightarrow \R_+$ be an ``error'' function. We say that
    $\hat{k}$ is an \textit{approximate universal kernel} on $\X$ if for all
    $\varepsilon>0$ and $f\in \mathcal{C}(\X)$, there exists $\hat h \in \hat H$
    such that $\|f-\hat h\|_\infty \leq \varepsilon + \rho(f)$.
\end{definition}

\subsection{Constructing a Smaller RKHS \texorpdfstring{$\hat
H$}{hatH}}\label{sec:build_Hat} 

In this section, we provide a principled method to ``sub-sample'' the sequence
$\varphi(x)$: \ukj{we will begin with a well-chosen finite family $(y_j) \in
\X^J$ and construct a well-suited basis $(x_n) \in \X^\N$ such that} a distance
sequence $(d_\X(x, x_n))_{n\in \N}$ is adequately approximable by the finite
number of distances $(d_\X(x, y_j))_{j\in \llbracket 1, J \rrbracket}$. We
illustrate this discretisation concept in \cref{fig:discretisation}.

\begin{figure}[H]
    \centering
    \includegraphics[width=0.4\linewidth]{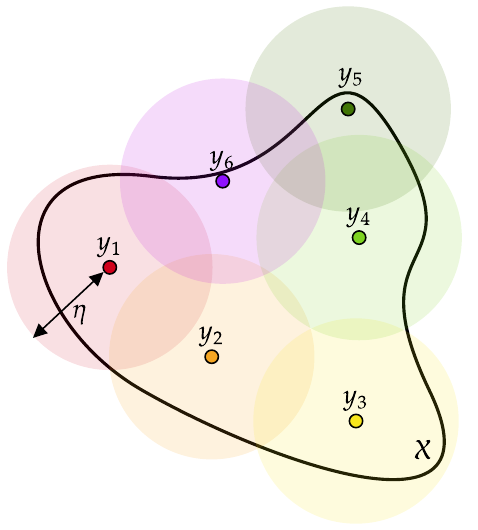}
    \caption{Discretisation of the space $\X$ into a cover of $J$ balls of
    radius $\eta>0$ centred at each $(y_j)_{j\in \llbracket 1, J \rrbracket}$.}
    \label{fig:discretisation}
\end{figure}

Instead of a basis of $\X$, we will now fix $\eta\in (0, \diamX]$ and consider
$(y_j)_{j \in \llbracket 1, J \rrbracket}$ a family of distinct points of $\X$
such that the family of (closed) balls $B_{d_\X}(y_j, \eta)$ covers $\X$. In
\cref{eqn:embed_X_RJ}, we introduce a map $\hat\varphi: \X \longrightarrow \R^J$
in the spirit of $\varphi$ defined in \cref{prop:inj_X_l2}, which we visualise
in \cref{fig:hatphi}.
\begin{equation}\label{eqn:embed_X_RJ}
    \hat\varphi:= \app{\X}{\R^J}{x}{\left(\cfrac{d_\X(x, y_j)}{\sqrt{J}}\right)_
    {j\in\llbracket 1, J \rrbracket}}.
\end{equation}

\begin{figure}[H]
    \centering
    \includegraphics[width=0.8\linewidth]{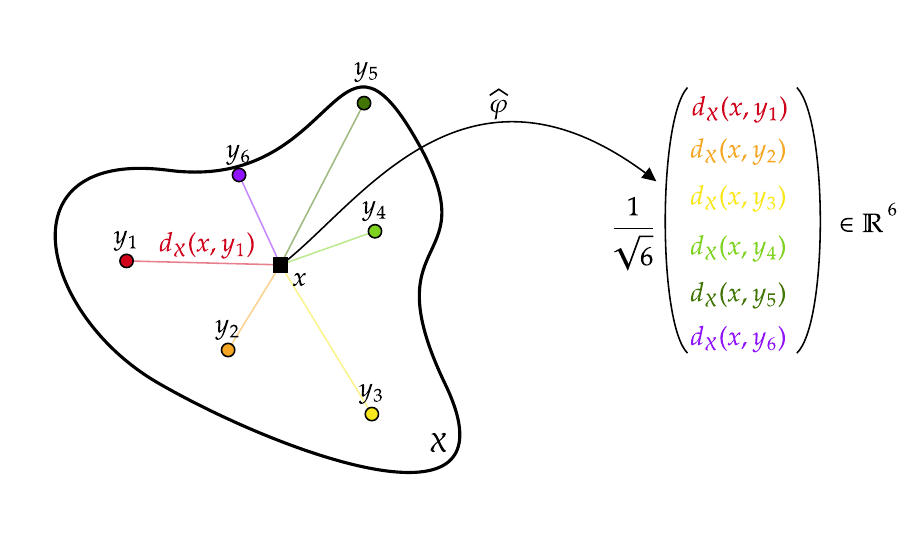}
    \caption{The mapping $\hat\varphi: \X \longrightarrow \R^J$ maps a point
    $x\in \X$ to the vector of normalised distances between $x$ and the centres
    $y_j$ of the covering.}
    \label{fig:hatphi}
\end{figure}
It is immediate to verify that $\hat\varphi: (\X, d_\X) \longrightarrow (\R^J,
\|\cdot\|_2)$ is 1-Lipschitz, thanks to the $J^{-1/2}$ normalisation. In
\cref{prop:B_isometry}, we show how to embed $\R^J\supset\hat\varphi(\X)$ into
$\ell^2$ with a mapping $B$, which will allow us to compare the RKHS induced by
$\hat\varphi$ and a particular $\varphi: \X\longrightarrow \ell^2$. To construct
$B$, we  first begin with a geometric series separation lemma, which will be
convenient to deal with the factor $\frac{1}{q^n}$ in $\varphi$.
\begin{lemma}\label{lemma:split_power_series} Let $J \geq 2,\; q \in (1, 1 +
    \frac{1}{J-1})$ and coefficients $(\lambda_1, \cdots, \lambda_J) \in (0,
    1)^J$ such that $\sum_j\lambda_j=1$, there exists $\alpha: \N
    \longrightarrow \llbracket 1, J \rrbracket$ with for all $j\in \llbracket 1,
    J \rrbracket,\; \alpha^{-1}(\{j\})$ infinite such that:
    \begin{equation}\label{eqn:alpha_split}
        \forall j \in \llbracket 1, J \rrbracket,\; 
        \sum_{n\in \alpha^{-1}(\{j\})}
        \cfrac{1}{q^n} = \lambda_j\cfrac{q}{q-1}.
    \end{equation}
\end{lemma}
\begin{proof}
    Set $S := \frac{q}{q-1}$. We will construct a sequence $(\alpha(N))_{N\in
    \N}$ by induction over $N$, verifying the property $$P_N: \forall j \in
    \llbracket 1, J \rrbracket,\; \sum_{n\in \llbracket 0, N\rrbracket\ :\
    \alpha(n)=j}\cfrac{1}{q^n} < \lambda_j S.$$ 
    
    \textit{Initialisation}: set $\alpha(0)$ the first $j\in \llbracket 1, J
    \rrbracket$ such that $1 < \lambda_j S$. Note that such a $j$ exists,
    otherwise summing over $j\in \llbracket 1, J\rrbracket$ yields 
    $$J \geq \cfrac{q}{q-1} > \cfrac{1 + \frac{1}{J-1}}{1 + \frac{1}{J-1} - 1} =
    J,$$ which is a contradiction. We have defined $\alpha(0):=j$ verifying
    $P_0$.

    \textit{Induction step}: let $N \in \N$, suppose $P_N$ true. We show that
    there exists $j\in \llbracket 1, J \rrbracket$ such that
    \begin{equation}\label{eqn:split_power_series_induction}
        \sum_{n\in \llbracket 0, N\rrbracket\ :\ \alpha(n)=j}\cfrac{1}{q^n} +
        \cfrac{1}{q^{N+1}} < \lambda_j S
    \end{equation}
    by contradiction. If that were not the case, we would have by summing
    \cref{eqn:split_power_series_induction} over $j \in \llbracket 1, J
    \rrbracket$:
    $$\sum_{n=0}^{N}\cfrac{1}{q^n} + \cfrac{J}{q^{N+1}} \geq S, $$ which by
    computation is equivalent to $q \geq 1 + \frac{1}{J-1}$, obtaining a
    contradiction. Selecting $j \in \llbracket 1, J \rrbracket$ such that
    \cref{eqn:split_power_series_induction} holds, we can set $\alpha(N+1) := j$
    which satisfies $P_{N+1}$.

    Now that $\alpha: \N \longrightarrow \llbracket 1, J \rrbracket$ verifying
    $(P_N)$ is constructed, we introduce the convergent series
    $$\forall j \in \llbracket 1, J \rrbracket,\; \forall N \in \N,\;  S_N^{(j)}
    := \sum_{n\in \llbracket 0, N \rrbracket\ : \ \alpha(n)=j}\cfrac{1}{q^n},\;
    S_\infty^{(j)} := \underset{N\longrightarrow +\infty}{\lim} S_N^{(j)}. $$
    Thanks to $(P_N)$, for all $j \in \llbracket 1, J \rrbracket$ taking the
    limit yields $S_\infty^{(j)}\leq \lambda_jS$, and summing over $j \in
    \llbracket 1, J \rrbracket$ gives $\sum_j S_\infty^{(j)} = S$, hence
    necessarily for all $j\in \llbracket 1, J \rrbracket,\; S_\infty^{(j)} =
    \lambda_j S$.
    
    Finally, observing the strict inequality in $P_N$ at each $N\in \N$ shows
    that $\alpha^{-1}(\{j\})$ has to be infinite, concluding the proof.
\end{proof}

We now turn to constructing an embedding $B: \R^J \longrightarrow \ell^2$, which
will allow us to compare $\hat\varphi$ and $\varphi$. An important property of
$B$ will be the correspondence between the inner products in $\R^J$ and $\ell^2$
(i.e. $B$ will be an isometry). The construction of this embedding revolves
around the construction of an adapted basis $(x_n)_{n\in \N}$ of $\X$ which is
balanced with respect to the covering by the balls $B(y_j, \eta)$, as
illustrated in \cref{fig:adapted_basis}.

\begin{figure}[H]
    \centering
    \includegraphics[width=0.8\linewidth]{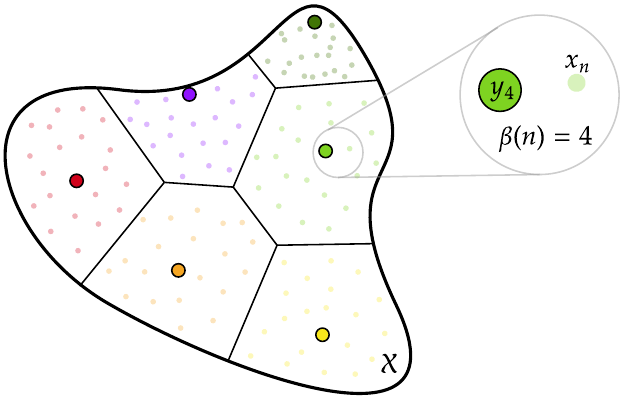}
    \caption{The basis $(x_n)_{n\in \N}$ is such that there equally as many
    $(x_n)$ in each region $\X_j$ of points closest to $y_j$. In the figure, we
    observe a zoom on the region $\X_4$, where the example point $x_n$ is
    closest to $y_4$. In mathematical terms, we write this property as $\beta(n)
    = y_j$, and in \cref{prop:adapted_basis} we will construct $(x_n)$ such that
    the sum $\sum_nq^{-2n}$ is split evenly between the sets
    $\beta^{-1}(\{j\})$.}
    \label{fig:adapted_basis}
\end{figure}

\begin{prop}\label{prop:adapted_basis} Take $q \in (1, \sqrt{1 +
    \frac{1}{J-1}})$. There exists a basis $(x_n)_{n\in \N}$ of $\X$ and a
    mapping $\beta: \N \longrightarrow \llbracket 1, J \rrbracket$ with infinite
    pre-images which verifies $\forall n \in \N,\; d_\X(x_n, y_{\beta(n)}) =
    \min_j d_X(x_n, y_j)$, and with the following property:
    \begin{equation}\label{eqn:balanced_basis}
        \forall j \in \llbracket 1, J \rrbracket,\; 
        \sum_{n\in \beta^{-1}(\{j\})}
        \cfrac{1}{q^{2n}} = \cfrac{1}{J}\cfrac{q^2}{q^2-1}.
    \end{equation}
\end{prop}
\begin{proof}
    Consider for $j\in \llbracket 1, J \rrbracket$ the set $\X_j := \{x\in \X :
    \argmin_{m}\ d_\X(x, y_m) = j\}$ (with disambiguation by taking the smallest
    minimiser if multiple exist). By definition, the sets $\X_j$ are disjoint
    and cover $\X$. Since $(\X, d_\X)$ is a compact metric space, each subset
    $\X_j$ is separable, allowing us to choose a basis $(z_n^{(j)})_{n\in \N}$
    of $\X_j$ for each $j\in \llbracket 1, J \rrbracket$. By
    \cref{lemma:split_power_series}, we can choose $\beta: \N \longrightarrow
    \llbracket 1, J \rrbracket$ with infinite pre-images which verifies
    \cref{eqn:balanced_basis}. Since for each $j\in \llbracket 1, J \rrbracket$,
    the set $\beta^{-1}(\{j\})\subset \N$ is infinite, we can choose $\omega_j:
    \beta^{-1}(\{j\}) \longrightarrow \N$ a bijection. We can now define
    $\forall n \in \N,\; x_n := z_{\omega_{\beta(n)}(n)}^{(\beta(n))}$, which is
    a basis of $\X$ since $\cup_j\X_j = \X$ and 
    $$\{x_n\}_{n\in \N} = \bigcup_j\
    \{z_{\omega_j(m)}^{(j)}\}_{m\in\beta^{-1}(\{j\})} = \bigcup_j\
    \{z_n^{(j)}\}_{n\in \N},$$ by construction. Furthermore, by definition, we
    have $\forall n \in \N,\; \argmin_{j}\ d_\X(x_n, y_j) = \beta(n)$, which
    shows that the mapping $\beta$ satisfies the desired properties.
\end{proof}
Using the adapted basis from \cref{prop:adapted_basis}, we can finally construct
an isometry $B: \R^J \longrightarrow \ell^2$:
\begin{prop}\label{prop:B_isometry} Take a basis $(x_n)$ of $\X$ and $\beta: \N
    \longrightarrow \llbracket 1, J \rrbracket$ as in \cref{prop:adapted_basis}.
    The mapping $B$ defined below is an isometry:
    \begin{equation}
        B := \app{\R^J}{\ell^2}{(u_j)_{j=1}^J}{
        \left(c_B\cfrac{u_{\beta(n)}}{q^n}\right)_{n\in \N}},\; 
        c_B := \frac{\sqrt{J(q^2-1)}}{q}.
    \end{equation}
\end{prop}
\begin{proof}
    The mapping $B$ is clearly linear, and for $u, v\in \R^J$ we compute using
    \cref{eqn:balanced_basis}:
    \begin{align*}
        \left\langle B(u), B(v)\right\rangle_{\ell^2}
        = \sum_{n=0}^{+\infty}\cfrac{c_B^2}{q^{2n}}u_{\beta(n)}
        v_{\beta(n)}
        = c_B^2 \sum_{j=1}^J u_jv_j
        \sum_{n\in \beta^{-1}(\{j\})}\cfrac{1}{q^{2n}}
        = c_B^2 \cfrac{1}{J}\cfrac{q^2}{q^2-1}
        \left\langle u, v \right\rangle_{\R^J} 
        = \left\langle u, v \right\rangle_{\R^J},
    \end{align*}
    which shows that $B$ is an isometry.
\end{proof}

In the following, we draw a correspondence between a RKHS $\hat H$ built with
$\hat\varphi$ from \cref{eqn:embed_X_RJ} and another RKHS $H$ built using
$\varphi$ from \cref{prop:inj_X_l2}. Let $U := \hat\varphi(\X)$, which is a
compact subset of $\R^J$, then let $\hat V := B(U)$, it is a compact subset of
$\ell^2$. Consider the injection $\varphi$ introduced in \cref{prop:inj_X_l2}
with basis $(x_n)$ and scale $q$ as in \cref{prop:B_isometry}. Define $V :=
\varphi(\X),\; W := V\cup \hat V$, which are also compact subsets of $\ell^2$.
We now summarise our objects in the following diagram:
\begin{equation}\label{eqn:embed_RJ_l2_diagram}
    \begin{tikzcd}
        \X \arrow[r, "\varphi"] \arrow[d, "\hat\varphi"] 
        & V \subset W \subset \ell^2 \\
        U \subset \R^J \arrow[r, "B"'] & \hat V \subset W \subset \ell^2
    \end{tikzcd}
\end{equation} We fix a kernel $k_W: W^2 \longrightarrow \R$ which is of Taylor type or radial (see \cref{def:taylor_kernel,def:radial_kernel}) and thus in particular universal on $W$, and introduce its canonical feature map:
\begin{equation}\label{eqn:def_k_W}
    \Phi_W := \app{W}{H_W}{u}{k_W(\cdot, u)},
\end{equation}
where $H_W = \oll{\Span}\left\{k_W(\cdot, u),\; u \in W\right\}\subset
\mathcal{C}(W)$ is the unique RKHS associated to the kernel $k_W$ \cite[Theorem
4.21]{christmann2008support}. Consider the kernels $k, \hat k$ on $\X$ defined
respectively as:
\begin{equation}\label{eqn:def_hat_k}
    k := \app{\X^2}{\R}{(x, y)}{k_W(\varphi(x),
    \varphi(y))},\quad 
    \hat k := \app{\X^2}{\R}{(x, y)}{k_W(B\circ\hat\varphi(x),
    B\circ\hat\varphi(y))}.
\end{equation}
By definition of the feature pair $(H_W, \Phi_W)$ for $k_W$, we observe that for
$x, y \in \X$:
\begin{equation}\label{eqn:k_k_hat_features}
    k(x, y) = \langle\Phi_W\circ\varphi(x), 
    \Phi_W\circ\varphi(y)\rangle_{H_W},\; 
    \hat k(x, y) = \langle\Phi_W\circ B\circ\hat\varphi(x), 
    \Phi_W\circ B\circ\hat\varphi(y)\rangle_{H_W}.     
\end{equation}
The RKHS spaces $H, \hat H$ associated to $k, \hat k$ are both subspaces of
$\mathcal{C}(\X)$ and can be written with the following respective feature pairs
$(H_W, \Phi), (H_W, \hat\Phi)$ (use \cref{eqn:k_k_hat_features} with
\cite[Theorem 4.21]{christmann2008support}):
\begin{align}
    H &= \left\{x\longmapsto\langle h_W, \Phi_W \circ \varphi(x)
    \rangle_{H_W},\; h_W \in H_W\right\},\; 
    \Phi := \app{\X}{H_W}{x}{\Phi_W \circ \varphi(x)}\label{eqn:H_W_H} \\
    \hat H &= \left\{x\longmapsto\langle h_W, \Phi_W \circ B \circ 
    \hat\varphi(x)\rangle_{H_W},\; h_W \in H_W\right\},\;
    \hat\Phi := \app{\X}{H_W}{x}{\Phi_W \circ B \circ \hat\varphi(x)}.
    \label{eqn:H_W_H_hat}
\end{align}
Notice that the feature space $H_W$ is shared. To finish the diagram, we
introduce the ``feature-to-map'' functionals:
\begin{equation}\label{eqn:feature_to_map}
    \Psi := \app{H_W}{H}{h_W}{x \mapsto \langle h_W, \Phi(x) \rangle_{H_W}},\;
    \hat\Psi := \app{H_W}{\hat H}{h_W}{x \mapsto \langle h_W, \hat\Phi(x)
    \rangle_{H_W}}.
\end{equation}
\ukj{By \cref{eqn:H_W_H,eqn:H_W_H_hat}, $\Psi$ and $\hat\Psi$ are surjective.}
Extending the diagram in \cref{eqn:embed_RJ_l2_diagram}, we obtain:
\begin{equation}
    \begin{tikzcd}
        & & V\subset W \subset \ell^2 \arrow[rd, "\Phi_W"] & & H \\
        \X \arrow[rru, "\varphi"] \arrow[rd, "\hat\varphi"] 
        \arrow[rrr, bend left=80, "\Phi"] 
        \arrow[rrr, bend right=80, "\hat\Phi"'] 
        & & & H_W \arrow[ru, "\Psi"] \arrow[rd, "\hat \Psi"] & \\
        & U\subset \R^J \arrow[r, "B"'] 
        & \hat V \subset W \subset \ell^2 \arrow[ru, "\Phi_W"] & & \hat H
    \end{tikzcd}
\end{equation}
Using the inner product correspondence induced by the isometry $B$ from
\cref{prop:B_isometry}, a tractable formula for $\hat k$ is obtained immediately
for Taylor and radial kernels. 
\begin{corollary}\label{cor:tractable_hat_k} The kernel $\hat k$ on $\X$ is
    given by, for all $ x, y \in \X$:
    \begin{itemize}
        \item if $k_W$ is a Taylor kernel (\cref{def:taylor_kernel}):
        \begin{equation}\label{eqn:tractable_hat_k_Taylor}
            \hat k(x, y) 
            = K(\langle \hat \varphi(x), \hat \varphi(y)\rangle_{\R^J}) 
            = \sum_{n=0}^{+\infty}a_n 
            \left(\cfrac{1}{J}\sum_{j=1}^Jd_\X(x, y_j)d_\X(y, y_j)\right)^n;
        \end{equation}
        \item if $k_W$ is a radial kernel (\cref{def:radial_kernel}):
        \begin{equation}\label{eqn:tractable_hat_k_radial}
            \hat k(x, y) 
            = K(\|\hat \varphi(x)-  \hat \varphi(y)\|_{\R^J}^2)
            = \int_{0}^{+\infty}\exp\left(-\cfrac{s}{J}
            \sum_{j=1}^J(d_\X(x, y_j)-d_\X(y, y_j))^2\right)\dd\mu(s).
        \end{equation}
    \end{itemize}
\end{corollary}
\begin{proof}
    Let $x, y\in \X$, we remind that from \cref{eqn:def_hat_k} that $\hat k (x,
    y) := k_W(B\circ\hat \varphi(x), B\circ\hat\varphi(y))$. Now by
    \cref{prop:B_isometry}, $B$ is an isometry, yielding:
    $$\langle B\circ\hat \varphi(x), B\circ\hat\varphi(y) \rangle_{\ell^2} =
    \langle \hat \varphi(x), \hat\varphi(y)\rangle_{\R^J};\quad \|B\circ\hat
    \varphi(x)- B\circ\hat\varphi(y)\|_{\ell^2}^2 = \|\hat \varphi(x)-
    \hat\varphi(y)\|_{\R^J}^2.$$ 
    \cref{eqn:tractable_hat_k_Taylor,eqn:tractable_hat_k_radial} are then
    obtained by replacing $k_W$ and $K$ by their definitions in the Taylor and
    radial cases.
\end{proof}
We refer to the expressions in
\cref{eqn:tractable_hat_k_Taylor,eqn:tractable_hat_k_radial} as ``tractable''
since they can be computed explicitly on a computer or approximated efficiently
to numerical precision (note that the measure $\mu$ in the radial kernel can be
discrete with finite support). \ukj{The numerical computation of $\hat\varphi$
is explicit and tractable in the sense that it can be done with a finite amount
of closed-form expressions. For the Taylor kernel, the infinite series can be
approximated to numerical precision, which we also refer to as ``tractable'', as
would be said of the exponential function for instance.}

\subsection{Showing that \texorpdfstring{$\hat H$}{hatH} is Approximately
Universal}\label{sec:approx_univ_Hat}

In this section, we show that the RKHS $\hat H$ introduced in
\cref{sec:build_Hat} is approximately universal on $\X$. We use the notation and
objects constructed in \cref{sec:build_Hat} extensively, \ukj{in particular, the
mapping $\varphi: \X \longrightarrow \ell^2$ is defined using
\cref{prop:inj_X_l2} with a suitable basis $(x_n)$ and scale $q$ from
\cref{prop:adapted_basis}.} The first approximation result we will show concerns
a comparison in $\ell^2$ between $\varphi(x)$ and $B\circ\hat\varphi(x)$:
\begin{prop}\label{prop:approx_l2} For $x\in \X$, we have $\|\varphi(x) -
    B\circ\hat\varphi(x)\|_{\ell^2} \leq \eta$. 
    
    The diameter of $W$ verifies $\diamW := \diam(W) \leq \ukj{2}\diamX$.
\end{prop}
\begin{proof}
    Let $n\in \N$, we look at the terms of the sequences $\varphi(x),
    B\circ\hat\varphi(x) \in W\subset\ell^2$:
    \begin{align*}
        \left|[\varphi(x)]_n - [B\circ\hat\varphi(x)]_n\right| &= 
        \left|\cfrac{c_\varphi d_\X(x, x_n)}{q^n} 
        - \cfrac{c_Bd_\X(x, y_{\beta(n)})}{\sqrt{J}q^n}\right|
        \leq \cfrac{1}{q^n}c_\varphi d_\X(x_n, y_{\beta(n)}).
    \end{align*}
    By construction of the covering $(B_{d_\X}(y_j, \eta))_j$ and of $\beta$
    (see \cref{prop:adapted_basis}), $d_\X(x_n, y_{\beta(n)}) \leq \eta$.
    Summing the squares over $n\in \N$ and replacing $c_\varphi =
    \frac{c_B}{\sqrt{J}}$ yields:
    $$\|\varphi(x) - B\circ\hat\varphi(x)\|_{\ell^2}^2 \leq \eta^2
    \sum_{n=0}^{+\infty}\cfrac{c_\varphi^2}{q^{2n}}=\eta^2.$$ For the diameter
    of $W := \varphi(\X)\cup (B\circ\hat\varphi(\X))$, we have by
    1-Lipschitzness of $\varphi, \hat\varphi$ and $B$: $\diam(\varphi(\X)) \leq
    \diamX$ and $\diam(B\circ\hat\varphi(\X)) \leq \diamX$. Using the inequality
    in the above display and the fact that $\eta\leq\diamX$, we conclude: 
    $$\ukj{\diamW = \max(\diam(\varphi(\X)),\ \sup_{x,y\in \X}\|\varphi(x) -
    B\circ\hat\varphi(y)\|_{\ell^2}) \leq \max(\diamX, 2\eta) \leq 2\diamX}.$$
    \vskip -30pt
\end{proof}

Using regularity properties of Taylor and radial kernels, we will show that the
kernel $\hat k$ is approximately universal on $\X$ by relating it to $k$ which
is universal by \cref{thm:univ_kernel_X}. First, we see in
\cref{lemma:Phi_W_Holder} that the canonical feature map $\Phi_W$ is
Hölder-continuous for Taylor kernels, and Lipschitz for radial kernels. We
introduce the radius of $W$: $\radW := \max_{w\in W}\|w\|_{\ell^2}$. Using the
definition of $W$ and of $\varphi,\hat\varphi$ and $B$ with their well-chosen
normalisations, it is easy to see that $\radW \leq \diamX$.
\begin{lemma}\label{lemma:Phi_W_Holder} The feature map $\Phi_W: (W,
    \|\cdot\|_{\ell^2}) \longrightarrow (H_W, \|\cdot\|_{H_W})$ has the
    following regularity:
    \begin{itemize}
        \item If $k_W$ is a Taylor kernel, then $\Phi_W$ is $\frac{1}{2}$-Hölder
        continuous:
        $$\forall u,v \in W,\; \|\Phi_W(u) - \Phi_W(v)\|_{H_W} \leq
        \sqrt{2\diamX C_{K'}}\|u-v\|_{\ell^2}^{\frac{1}{2}},$$ where $C_{K'} :=
        \max_{t\in [-\ukj{4}\diamX^2, \ukj{4}\diamX^2]} |K'(t)|$.
        \item If $k_W$ is a radial kernel, then $\Phi_W$ is
        $\sqrt{2C_{K'}}$-Lipschitz:
        $$\forall u,v \in W,\; \|\Phi_W(u) - \Phi_W(v)\|_{H_W} \leq
        \sqrt{2C_{K'}}\|u-v\|_{\ell^2}, $$ where $C_{K'} := \max_{t\in [0,
        \ukj{4}\diamX^2]}|K'(t)|$.
    \end{itemize}
\end{lemma}
\begin{proof}
    First, we remind that by \cref{prop:approx_l2}, we have $\diam(W) \leq
    \ukj{2}\diamX$. For the proof, we take inspiration from \cite[Section
    4.2]{fiedler2023lipschitz}. Using the reproducing property, we begin
    computations for both kernel types, letting $u,v\in W$:
    \begin{align*}
        \|\Phi_W(u) - \Phi_W(v)\|_{H_W}^2
        &= k_W(u,u) - 2k_W(u,v) + k_W(v,v) \\
        &\leq|k_W(u,u) - k_W(u,v)| + |k_W(v,v) - k_W(u,v)|.
    \end{align*}
    For Taylor kernels, we use the fact that $K$ is $C_{K'}$-Lipschitz on
    $[-\ukj{4}\diamX^2, \ukj{4}\diamX^2]$ and the Cauchy-Schwarz inequality for
    $\langle \cdot, \cdot \rangle_{\ell^2}$:
    \begin{align*}
        \|\Phi_W(u) - \Phi_W(v)\|_{H_W}^2
        &\leq C_{K'}\left(|\langle u, u \rangle_{\ell^2} - \langle u, v
        \rangle_{\ell^2}| + |\langle v, v \rangle_{\ell^2} - \langle u,
        v\rangle_{\ell^2}|\right) \\
        &\leq C_{K'}(\|u\|_{\ell^2} + \|v\|_{\ell^2})
        \|u-v\|_{\ell^2} \\
        &\leq 2\radW C_{K'}\|u-v\|_{\ell^2},
    \end{align*}
    and we conclude using $\radW\leq\diamX$. For radial kernels, we use the fact
    that $K$ is $C_{K'}$-Lipschitz on $[0, \diamW^2]$ (we remind that $K$ is
    non-increasing on $[0, +\infty)$):
    \begin{align*}
        \|\Phi_W(u) - \Phi_W(v)\|_{H_W}^2
        = 2(K(0) - K(\|u-v\|_{\ell^2}^2))
        \leq 2C_{K'}\|u-v\|_{\ell^2}^2.
    \end{align*}
    \vskip-20pt
\end{proof}
We now use \cref{lemma:Phi_W_Holder} to approximate any $h \in H$ with a $\hat h
\in \hat H$ with a certain error, which we approach by comparing the
feature-to-map functionals $\Psi$ and $\hat\Psi$ from
\cref{eqn:H_W_H,eqn:H_W_H_hat}. \newcommand{\rhosentence}{$\rho_0 =
\eta^{\frac{1}{2}}\sqrt{2\diamX C_{K'}}$ for a Taylor kernel and $\rho_0 =
\eta\sqrt{2C_{K'}}$ for a radial kernel.}
\begin{prop}\label{prop:approx_H_H_hat} For $h \in H$, take $h_W \in H_W$ such
    that $h = x \longmapsto \langle h_W, \Phi(x)\rangle_{H_W} = \Psi(h_W)$. Then
    let $\hat h := x \longmapsto \langle h_W, \hat\Phi(x)\rangle_{H_W} =
    \hat\Psi(h_W)$. Denoting $\|\cdot\|_{\infty}$ the supremum norm on $\X$, we
    have:
    \begin{equation}\label{eqn:approx_H_H_hat}
        \|h-\hat h\|_{\infty} \leq \rho_0\|h_W\|_{H_W},
    \end{equation}
    where \rhosentence.     
\end{prop}
\begin{proof}
    First, we use the regularity of $\Phi_W$ from \cref{lemma:Phi_W_Holder}: we
    have for $x\in X$,
    \begin{align*}
        |h(x) - \hat h(x)| &= \langle h_W, \Phi(x)
        - \hat\Phi(x)\rangle_{H_W}
        \leq \|h_W\|_{H_W}\|\Phi(x) - \hat\Phi(x)\|_{H_W} \\
        &= \|h_W\|_{H_W}\|\Phi_W\circ\varphi(x) - 
        \Phi_W\circ B\circ\hat\varphi(x)\|_{H_W} \\
        &\leq \tilde{c}\|h_W\|_{H_W}\|\varphi(x) - 
        B\circ\hat\varphi(x)\|_{\ell^2}^s,
    \end{align*}
    where $(\tilde{c}, s)=(\sqrt{2\diamX C_{K'}}, \frac{1}{2})$ for a Taylor
    kernel and $(\tilde{c}, s) = (\sqrt{2C_{K'}}, 1)$ for a radial kernel.
    Combining with \cref{prop:approx_l2}, we obtain \cref{eqn:approx_H_H_hat}.
\end{proof}
Using the universality of the kernel $k$ (thanks to \cref{thm:univ_kernel_X}),
we can frame the result of \cref{prop:approx_H_H_hat} as an approximate
universality property of $\hat k$. Again, the approximation error functions
depend on the type of kernel $k_W$.
\begin{theorem}\label{thm:approx_univ_kernel} Let $\varepsilon>0$ and $f \in
    \mathcal{C}(\X)$, the element $h[\varepsilon, f] \in H$ defined by
    \begin{equation}\label{eqn:h_eps_f}
        h[\varepsilon, f] := 
        \underset{h \in H : \|h - f\|_\infty \leq \varepsilon}
        {\argmin}\ \|h\|_{H}^2        
    \end{equation}
    is well-defined, and there exists $\hat h \in \hat H$ such that:
    \begin{equation}\label{eqn:approx_univ_kernel}
        \|f - \hat h\|_{\infty} \leq \varepsilon +
        \rho_0 \|h[\varepsilon, f]\|_{H},
    \end{equation}
    where \rhosentence.
\end{theorem}
\begin{proof}
    First, we introduce:
    \begin{equation*}%
        h_W[\varepsilon, f] := 
        \underset{h_W \in H_W : \|\Psi(h_W) - f\|_\infty \leq \varepsilon}
        {\argmin}\ \|h_W\|_{H_W}^2        
    \end{equation*}
    We show that $h_W[\varepsilon, f]$ and $h[\varepsilon, f]$ are well-defined.
    The triangle inequality ensures that the sets $\mathcal{B}_{H_W} := \{h_W
    \in H_W : \|\Psi(h_W) - f\|_\infty \leq \varepsilon\}$ and $\mathcal{B}_H :=
    \{h \in H : \|h - f\|_\infty \leq \varepsilon\}$ are convex. 
    
    \ukj{We now show the continuity of $\Psi$ as a mapping $(H_W,
    \|\cdot\|_{H_W}) \longrightarrow (\mathcal{C}(\X), \|\cdot\|_\infty)$.
    Fixing $x\in \X$, we upper-bound by the Cauchy-Schwarz inequality:
    \begin{equation}\label{eqn:c-s_inequality_Psi}
        |\Psi(h_W)[x]| = |\langle h_W, \Phi(x)\rangle_{H_W}|
        \leq \|h_W\|_{H_W}\|\Phi(x)\|_{H_W},        
    \end{equation}
    then we use the definition $\Phi = \Phi_W \circ \varphi$ to show the
    continuity of $\Phi : (\X, d_\X) \longrightarrow (H_W, \|\cdot\|_{H_W})$: by
    \cref{prop:inj_X_l2}, $\varphi: (\X, d_\X) \longrightarrow (W,
    \|\cdot\|_{\ell^2})$ is continuous, and by \cref{lemma:Phi_W_Holder},
    $\Phi_W: (W, \|\cdot\|_{\ell^2}) \longrightarrow (H_W, \|\cdot\|_{H_W})$ is
    also continuous. Combining \cref{eqn:c-s_inequality_Psi} with the continuity
    of $\Phi$ and the compactness of $\X$ ensures that there exists $C>0$
    independent of $x\in \X$ and $h_W\in H_W$ such that $|\Psi(h_W)[x]|\leq
    C\|h_W\|_{H_W}$, thus $\Psi$ as a mapping $(H_W, \|\cdot\|_{H_W})
    \longrightarrow (\mathcal{C}(\X), \|\cdot\|_\infty)$ is continuous.}

    \ukj{Thanks to the continuity of $\Psi: (H_W, \|\cdot\|_{H_W})
    \longrightarrow (\mathcal{C}(\X), \|\cdot\|_\infty)$, we conclude that
    $\mathcal{B}_{H_W} := \{h_W \in H_W : \|\Psi(h_W) - f\|_\infty \leq
    \varepsilon\}$ is closed in $(H_W, \|\cdot\|_{H_W})$. Regarding
    $\mathcal{B}_H$, by \cite[Lemma 4.23]{christmann2008support}, since the
    kernel $k$ is bounded on $\X$, the inclusion $\iota: (H, \|\cdot\|_{H})
    \longrightarrow (\mathcal{C}(\X), \|\cdot\|_\infty)$ is continuous. We
    deduce the closedness of $\mathcal{B}_H =
    \iota^{-1}\left(\mathcal{B}_{\mathcal{C}(\X)}(f, \varepsilon)\right)$ in
    $(\mathcal{C}(\X), \|\cdot\|_\infty)$.}

    Finally, the sets $\mathcal{B}_{H_W}$ and $\mathcal{B}_{H}$ are
    non-empty since $H = \Psi(H_W)$ is dense in $(\mathcal{C}(X),
    \|\cdot\|_\infty)$. 
    
    We conclude that $\mathcal{B}_{H_W}$, resp. $\mathcal{B}_{H}$ is a non-empty
    closed convex set in the Hilbert space $(H_W, \|\cdot\|_{H_W})$, resp. $(H,
    \|\cdot\|_{H})$, and the Hilbert projection theorem (or directly
    \cite[Theorem 4.10]{rudin1987real}) ensures that $h_W[\varepsilon, f]$,
    resp. $h[\varepsilon, f]$ is uniquely defined. 

    Now, we show that $\|h[\varepsilon, f]\|_H = \|h_W[\varepsilon, f]\|_{H_W}$.
    By \cite[Theorem 4.21]{christmann2008support}, we have for all $h\in H$:
    $$\|h\|_H = \inf\{\|h_W\|_{H_W},\; h=\Psi(h_W)\}.$$ By the same argument as
    before (using \cite[Theorem 4.10]{rudin1987real}), we show that the infimum
    is attained. The equality between norms is then straightforward by
    separating both inequalities and using $H = \Psi(H_W)$.
    
    To obtain \cref{eqn:approx_univ_kernel}, we take $h := \Psi(h_W[\varepsilon,
    f])$ in \cref{eqn:approx_H_H_hat} and $\hat h := \hat\Psi(h_W[\varepsilon,
    f]) \in \hat H$, and apply the triangle inequality for $\|\cdot\|_\infty$,
    using $\|h-f\|_\infty\leq\varepsilon$ and $\|h[\varepsilon, f]\|_H =
    \|h_W[\varepsilon, f]\|_{H_W}$.
\end{proof}
The approximation result in \cref{eqn:approx_univ_kernel} shows that $\hat k$ is
$\rho$-approximately universal (\cref{def:approximate_universal_kernel}) for
$\rho(f) := \rho_0\|h[\varepsilon, f]\|_{H}$. In the case where $\X$ is of
dimension $d$ (or has intrinsic dimension $d$), the number of covering balls
scales as $J=\mathcal{O}(\eta^{-d})$, which does not impact the approximation
rate, as is commonly the case in kernel methods which do not suffer from the
curse of dimensionality (see for example \cite[Section 4.1]{gretton2012kernel}).
However, as is typically the case for discretisation methods, the rate
$J=\mathcal{O}(\eta^{-d})$ is computationally prohibitive for small
discretisation step $\eta$ in high dimension $d$.

From a functional standpoint, a larger oscillation (a large value for $C_{K'} =
\max_{t\in [-\diamX^2, \diamX^2]}|K'(t)|$ e.g. for the Taylor case), of the
function $K$ worsens the error, which could be understood as excessive locality
or over-fitting. Finally, the error term $\rho(f)$ is relative in the sense that
it depends on $\|h[\varepsilon, f]\|_{H}$, which is the smallest possible norm
of an $\varepsilon$-approximation of $f$ within $H$, and can be seen as a
measure of complexity of $f$ (in loose terms). This term depends on $q$, and
while the exact dependence is unclear, we expect it to grow as $q$ increases.

\subsection{An Approximate Universal Truncated Kernel}\label{sec:truncated_kernels}

In this section, we consider another approximate universal kernel which is
obtained by truncation of $\varphi$. \ukj{We will undergo a similar process as
the construction of $\hat H$ in \cref{sec:build_Hat} and follow closely the
proof methods of \cref{sec:approx_univ_Hat}}.

A natural idea is to simply consider a truncation of the mapping $\varphi$ from
\cref{prop:inj_X_l2}: fixing a basis $(x_n)_{n\in \N}$ of $\X$, a discretisation
size $N\geq 2$ and scale $q>1$, consider the mapping:
\begin{equation*}
    \varphi_t := \app{\X}{\R^N}{x}{\left(\cfrac{c_\varphi d_\X(x, x_n)}{q^j}
    \right)_{n \in \llbracket 0, N-1 \rrbracket}}.
\end{equation*}
Straightforward computation shows that $\varphi_t: (\X, d_\X) \longrightarrow
(\ell^2, \|\cdot\|_{\ell^2})$ is $\sqrt{1-q^{-2N}}$-Lipschitz. We introduce the
``padding'' isometry:
\begin{equation*}
    B_t := \app{\R^N}{\ell^2}{(u_n)_{n=0}^{N-1}}
    {(u_0, \cdots, u_{N-1}, 0, \cdots)},
\end{equation*}
Similarly to \cref{sec:build_Hat}, we take $V:=\varphi(\X),\; U_t :=
\varphi_t(\X)$ and $V_t := B_t(U_t)$, allowing us to introduce the compact set
$W :=V\cup V_t \subset \ell^2$ (we use the same notation as in
\cref{sec:build_Hat} to alleviate notation). Take $k_W$ a Taylor or radial
kernel on $W$, and introduce the kernel:
\begin{equation*}
    k_t := \app{\X^2}{\R}{(x, y)}{k_W(B_t\circ\varphi_t(x), 
    B_t\circ\varphi_t(y))}.
\end{equation*}
We continue with the feature pair $(H_W, \Phi_t)$ for the RKHS $H_t$ associated
to $k_t$, where:
\begin{equation*}
    \Phi_t := \app{\X}{H_W}{x}{\Phi_W\circ B_t \circ \varphi_t(x)}.
\end{equation*}
\ukj{As in \cref{eqn:feature_to_map} we introduce the ``feature-to-map''
functionals:
\begin{equation}\label{eqn:feature_to_map_t}
    \Psi := \app{H_W}{H}{h_W}{x \mapsto \langle h_W, \Phi(x) \rangle_{H_W}},\;
    \Psi_t := \app{H_W}{H_t}{h_W}{x \mapsto \langle h_W, \Phi_t(x)
    \rangle_{H_W}},
\end{equation}}
and finish the diagram:
\begin{equation*}
    \begin{tikzcd}
        & & V\subset W \subset \ell^2 \arrow[rd, "\Phi_W"] & & H \\
        \X \arrow[rru, "\varphi"] \arrow[rd, "\varphi_t"] 
        \arrow[rrr, bend left=80, "\Phi"] 
        \arrow[rrr, bend right=80, "\Phi_t"'] 
        & & & H_W \arrow[ru, "\Psi"] \arrow[rd, "\Psi_t"] & \\
        & U_t\subset \R^N \arrow[r, "B_t"'] 
        & V_t \subset W \subset \ell^2 \arrow[ru, "\Phi_W"] & & H_t
    \end{tikzcd}
\end{equation*}
The computation in the proof of \cref{cor:tractable_hat_k} stands, but the
coefficients in the expression of $\varphi_t$ lead to a different expression for
$k_t$, which is a truncated version of $k$: if $k_W$ is a Taylor kernel, we
have:
\begin{equation*}
    k_t(x, y) = K\left(\langle \varphi_t(x), \varphi_t(y)\rangle_{\R^N}\right) 
    = \sum_{n=0}^{+\infty}a_n\left(\sum_{m=0}^{N-1}
    \cfrac{c_\varphi^2 d_\X(x, x_m)d_\X(y, x_m)}{q^{2m}}\right)^n,
\end{equation*}
and likewise for radial kernels:
\begin{equation*}
    k_t(x, y) = K\left(\|\varphi_t(x) - \varphi_t(y)\|_{\R^N}^2\right)
    = \int_{0}^{+\infty}\exp\left(-s\sum_{n=0}^{N-1}
    \cfrac{c_\varphi^2(d_\X(x, x_n) - d_\X(y, x_n))^2}{q^{2n}}\right)\dd\mu(s).
\end{equation*}
We now adapt \cref{prop:approx_l2} to $k_t$:
\begin{prop}\label{prop:approx_l2_truncation}
    For $x \in \X$, we have $\|\varphi(x) - B_t\circ\varphi_t(x)\|_{\ell^2} \leq
    \frac{\diamX}{q^N}.$ 
    
    The diameter of $W$ verifies $\diamW \leq \ukj{2}\diamX$.
\end{prop}
\begin{proof}
    For $n \in \llbracket 0, N-1 \rrbracket,$ by construction $[\varphi(x)]_n =
    [B_t\circ\varphi_t(x)]_n$. For $n\geq N$, we have:
    \begin{align*}
        \left|[\varphi(x)]_n - [B_t\circ\varphi_t(x)]_n\right| 
        = \cfrac{c_\varphi d_\X(x, y_n)}{q^n},
    \end{align*}
    and by bounding the distance term by $\diamX$, and summing the squares,
    we obtain:
    \begin{align*}
        \|\varphi(x) - B_t\circ\varphi_t(x)\|_{\ell^2}^2 
        \leq \sum_{n=N}^{+\infty}\cfrac{c_\varphi^2 \diamX^2}{q^{2n}} 
        = \cfrac{c_\varphi^2 \diamX^2 q^2}{q^{2N}(q^2 - 1)} 
        = \cfrac{\diamX^2}{q^{2N}}.
    \end{align*}
    As for the result on $\diamW$, it follows from 1-Lipschitzness as done in
    \cref{prop:approx_l2}.
\end{proof}
As in \cref{sec:approx_univ_Hat}, it is easy to verify that $\radW :=
\max_{w\in W}\|w\|_{\ell^2}\leq\diamX$. Following the same steps as in
\cref{thm:approx_univ_kernel}, we show a similar result for $k_t$, replacing
$\eta$ with $\diamX q^{-N}$:
\begin{theorem}\label{thm:approx_univ_kernel_truncation} Let $\varepsilon>0$ and
    $f \in \mathcal{C}(\X)$, there exists $h_t \in H_t$ such that:
    \begin{equation}\label{eqn:approx_univ_kernel_truncation}
        \|f - h_t\|_{\infty} \leq \varepsilon +
        \rho_t\|h[\varepsilon, f]\|_{H},
    \end{equation}
    where $\rho_t = q^{-N/2}\diamX\sqrt{2C_{K'}}$ for a Taylor kernel and
    $\rho_t = q^{-N}\diamX\sqrt{2C_{K'}}$ for a radial kernel, with the
    constants $C_{K'}$ as in \cref{lemma:Phi_W_Holder}.
\end{theorem}
\begin{proof}
    \ukj{We follow the same progression as in the proof of
    \cref{thm:approx_univ_kernel}. First, we follow the proof of
    \cref{prop:approx_H_H_hat}, applying \cref{lemma:Phi_W_Holder}, then
    upper-bounding the term $\|\varphi(x)-B_t\circ\varphi_t(x)\|_{\ell^2}^2$
    using \cref{prop:approx_l2_truncation}, and the only difference is the
    replacement of the term $\eta$ with $\diamX q^{-N}$. Having adapted
    \cref{prop:approx_H_H_hat}, the proof of
    \cref{thm:approx_univ_kernel_truncation} follows as in
    \cref{thm:approx_univ_kernel}.}
\end{proof}
To compare with the rate from \cref{thm:approx_univ_kernel}, we see that the
term $\eta$ is replaced by $\diamX q^{-N}$. While $q^{-N}$ becomes rapidly
smaller as $q$ increases, we suspect the term $\|h[\varepsilon, f]\|_{H}$ to
grow quickly as $q$ increases, which would favour the kernel $\hat k$ from
\cref{sec:build_Hat}. \ukj{From an intuitive standpoint, the quality of the
truncation approximation depends on how well the truncated basis
$(x_n)_{n=0}^{N-1}$ represents $\X$. For example, if $\X$ is a manifold of
$\R^d$ with two connected components, and the first $N$ elements are all in the
first component, it can be expected that a substantial part of the information
about the space is lost, hindering function approximation. This issue can arise
because the ``basis'' property of $(x_n)$ relates to the full sequence, whereas
truncation focuses on the first $N$ terms, which are not assumed to satisfy
particular conditions. We saw in \cref{sec:build_Hat,sec:approx_univ_Hat} a more
principled discretisation approach, which is in some sense a refinement of the
truncation principle.}
\subsection*{Acknowledgements}

We thank Joan Glaunès for carefully proofreading this work and for his valuable
insight.

This research was funded in part by the Agence nationale de la recherche (ANR),
Grant ANR-23-CE40-0017 and by the France 2030 program, with the reference
ANR-23-PEIA-0004.

\printbibliography
\end{document}